\documentclass[10pt,final]{article}
\usepackage[english]{babel}

\usepackage[a4paper,top=43mm, bottom=18mm, left=25mm, right=25mm]{geometry} 
\usepackage{amsmath,amssymb}
\usepackage{amsthm}
  \usepackage{paralist}
  \usepackage{graphics} %% add this and next lines if pictures should be in esp format
  \usepackage{epsfig} %For pictures: screened artwork should be set up with an 85 or 100 line screen
\usepackage{graphicx} 
 \usepackage{epstopdf}%This is to transfer .eps figure to .pdf figure; please compile your paper using PDFLeTex or PDFTeXify.
 \usepackage[colorlinks=true]{hyperref}

\hypersetup{urlcolor=blue, citecolor=red}

\usepackage[sort]{cite}
\newtheorem{theorem}{Theorem}[section]
\newtheorem{corollary}{Corollary}
\newtheorem{lemma}[theorem]{Lemma}
\newtheorem{proposition}[theorem]{Proposition}

\newtheorem{definition}[theorem]{Definition}
\newtheorem{remark}{Remark}
\newtheorem{example}{Example}

\newcommand{\R}{\mathbb{R}}
\newcommand{\diag}{\text{diag}}
\newcommand{\1}{\mathbb{1}}

\newcommand{\B}{{\cal{B}}}
\newcommand{\A}{{\mathcal{A}}}
\newcommand{\KL}{\mathcal{K}^L}
\newcommand{\Ko}{\mathcal{K}^0}
\newcommand{\Too}{T_{1,1}}
\newcommand{\Tot}{T_{1,2}}
\newcommand{\Tto}{T_{2,1}}
\newcommand{\Ttt}{T_{2,2}}
\newcommand{\Aoo}{\A_{1,1}}
\newcommand{\Aot}{\A_{1,2}}
\newcommand{\Ato}{\A_{2,1}}
\newcommand{\Att}{\A_{2,2}}
\newcommand{\Boo}{\B_{1,1}}
\newcommand{\Bot}{\B_{1,2}}
\newcommand{\Bto}{\B_{2,1}}
\newcommand{\Btt}{\B_{2,2}}
\newcommand{\Koo}{\Ko_{1,1}}
\newcommand{\Kot}{\Ko_{1,2}}
\newcommand{\Kto}{\Ko_{2,1}}
\newcommand{\Ktt}{\Ko_{2,2}}
\newcommand{\KLoo}{\KL_{1,1}}
\newcommand{\KLot}{\KL_{1,2}}
\newcommand{\KLto}{\KL_{2,1}}
\newcommand{\KLtt}{\KL_{2,2}}

\usepackage{xcolor}
\usepackage{tikz}
\usepackage{pgfplots}
\newlength{\fwidth}

\begin{document}

\title{Boundary Stabilization with restricted observability
 }

\author{Mapundi Kondwani Banda\footnote{Department of Mathematics and Applied Mathematics, University of Pretoria, Private Bag X20, Hatfield 0028, Republic of South Africa, \texttt{mapundi.banda@up.ac.za}}, Jan Friedrich\footnote{Chair of Numerical Analysis, Institute for Geometry and Applied Mathematics, RWTH Aachen University, Templergraben 55, 52056 Aachen, Germany  \texttt{\{friedrich, herty\}@igpm.rwth-aachen.de}}, Michael Herty\footnotemark[2]}
%\author{Mapundi Kondwani Banda\footnotemark[1], Jan Friedrich\footnotemark[2], Michael Herty\footnotemark[2]}
%\address{\footnotemark[1] {Department of Mathematics and Applied Mathematics, University of Pretoria, Private Bag X20, Hatfield 0028, Republic of South Africa}\newline \footnotemark[2]{Chair of Numerical Analysis, Institute for Geometry and Applied Mathematics, RWTH Aachen University, Templergraben 55, 52056 Aachen, Germany  \texttt{\{friedrich, herty\}@igpm.rwth-aachen.de}}}
\date{\today}                                          
% Activate to display a given date or no date
%\maketitle
\maketitle
\begin{abstract}
Lyapunov functions are popularly used to investigate the stabilization problem of  systems of hyperbolic conservation laws with boundary controls.
In real life applications often not every boundary value can be observed.
In this work, we show the stabilization under a restricted boundary observability. 
Thereby, we apply the boundary control directly on the observed (physical) variables. 
Using well-known stabilization results from the literature, we also discuss examples such as a density flow model or the Saint-Venant equations. This shows that a restricted observation can result in more restrictive control choices or can prevent the system from stabilizing.
\end{abstract}
\medskip
%\keywords{
\medskip
\section{Introduction}
Hyperbolic systems of partial differential equations of balance laws are used to model many physical phenomena with finite speed of propagation.
They cover a wide range of applications, such as traffic flow, Maxwell's equations, Euler's equations, or Saint-Venant's equations, to name a few.
Of particular practical relevance is the boundary control of such systems. 
For example, the water flow of an open channel can be controlled by gates at both ends of the channel.
A suitable control can stabilize a system in an efficient manner. The most powerful property of these models is the existence of the so-called Riemann coordinates, which enable the proof of classical solutions as well as control. 
Stabilization to a steady state has been studied by several authors in the last decade, e.g. \cite{bastin2016stability,coron2008dissipative,banda2020numerical,bastin2008using,bastin2017quadratic,hayat2019quadratic,hayat2021exponential,hayat2021boundary,diagne2011lyapunov} and the references therein.
The desired feedback laws can be derived with the help of suitable Lyapunov functions such that the systems can stabilize exponentially in time in a suitable norm.

Most of the works study the physical system by performing a coordinate change to so-called Riemmann coordinates or directly in Riemann coordinates \cite{bastin2016stability,coron2008dissipative,banda2020numerical,bastin2008using,diagne2011lyapunov}.
Only a few deal with control laws in physical variables, and study them mostly for specific examples of systems, the reader may refer to \cite{bastin2017quadratic,hayat2019quadratic,hayat2021exponential} for examples. 
However, for practitioners, control laws in physical variables are the most relevant, since these are the variables they actually observe.
Obviously, with a further coordinate change control laws in Riemann coordinates can be transformed back to feedback laws in physical variables. The interest in boundary observation is due to the fact that distributed parameter systems are usually not available \cite{Castillo2013observer}. Sensors are usually installed at the boundaries.  In \cite{Castillo2013observer} sufficient conditions for infinite dimensional boundary observer design in the presence of static or dynamic boundary control were demonstrated. The asymptotic convergence of the estimation error is demonstrated by employing the Lyapunov function. In \cite{li2008observability} a direct and constructive approach based on the semi-global $C^1$ solution is used to establish the local exact boundary observability for one-dimensional first-order quasilinear hyperbolic systems with general nonlinear boundary conditions. Other presentations include \cite{li2011constructive,Li2008note,liu2019application,kitsos2021observer}. In \cite{gugat2023observer} a practical example on a networked domain is used to prove that the observer system converges to the original system state exponentially fast in the $L^2$-norm if the measurements are exact.
However, in practice not all variables may be measurable, such that the resulting control laws from Riemann coordinates might not be attainable. This problem was also highlighted in \cite{Li2008note} for one-side exact boundary observability. In this case a constructive method was used to demonstrate that a suitable coupling of the variables in the quasilinear hyperbolic system itself was needed for observability. In \cite{som2015application}, state estimation is used to estimate a fourth boundary variable in the case where three variables are known for nonlinear coupled distributed hyperbolic equations for non-conservative laws in a pressurized water pipe model. A Lipschitz property and a Lyapunov function are used to prove the exponential stability of the estimation error.
Therefore, it is of particular interest to investigate, in general, if it is still possible to derive a control law if boundary observability is restricted due to limited availability of boundary information.

We will study the effect of missing information of the physical boundary values in one space dimension by still going to Riemann coordinates for stabilization.
Therefore, in section \ref{sec:PrelResults}, we recall some important results from the literature and then investigate the question of stability or stabilization in more detail.
In particular, we will consider the case of linear controls of rank one, local controls, and consider as examples a density flow system {and the Saint-Venant equations}.

\section{Preliminary Results} \label{sec:PrelResults}
We will consider a system of hyperbolic linear conservation laws in one space dimension. Such systems can, for example, be obtained after linearizing a nonlinear system around a suitable steady state.
The evolution of the physical or so-called state variables $Y:\R_+\times\R\to \R^n$ is described by the equation
\begin{align}\label{eq:state_sys}
\partial_t Y(t,x) + M\partial_x Y(t,x)=0,\quad &t>0,\ x\in [0,L];\\
Y(0,x)=Y_0(x),\quad &x\in [0,L].\notag
\end{align}
Here $Y_0$ is a suitable initial condition, $x$ and $t$ are the space and time coordinates, respectively, and $M\in \mathbb{R}^{n\times n}$ is diagonalizable such that there exists a matrix $T\in\R^{n\times n}$ with $T^{-1}M T=\Lambda$ where $\Lambda$ is a diagonal matrix with the non-zero real eigenvalues of $M$ as its entries.
The eigenvalues are assumed to be distinct and ordered as $\lambda_1>\dots>\lambda_m>0>\lambda_{m+1}>\dots>\lambda_n$.
Hence, the system is strictly hyperbolic.
For a well-posed system on $[0,1]$, $n$ boundary conditions need to be specified.
In many physical examples the boundary conditions for \eqref{eq:state_sys} (after a possible linearization) can be formally described by 
\begin{equation}\label{eq:physicalbcU}
\A Y(t,0) + \B Y(t,L)= U(Y(t,0),Y(t,L)),
\end{equation}
where $\A \in\mathbb{R}^{n\times n}$ and $\B \in\mathbb{R}^{n\times n}$ contain the physical constraints on the boundaries and $U: \mathbb{R}^{n}\times \R^{n}\to \R^n$ is a possible nonlinear control law.
The system \eqref{eq:state_sys} can be rewritten using Riemann coordinates $R:\R_+\times\R\to \R^n$, where $R(t,x)=T^{-1}Y(t,x)$, in the form
\begin{align}\label{eq:Riemann_sys}
\partial_t R(t,x) + \Lambda\partial_x R(t,x)=0,&\quad t>0,\ x\in [0,L];\\
R(0,x)=R_0(x),&\quad x\in [0,L].
\end{align}
With $R^+\in\R^m$ and $R^-\in\R^{n-m}$, we denote the entries with positive and negative eigenvalues, respectively.
To close the system of equations \eqref{eq:Riemann_sys} on $[0, L]$, boundary conditions are also necessary.
The only boundary conditions that one can impose are those corresponding to information entering the domain.
Thus at $x = 0$ the boundary conditions must be imposed for $R^+$
while at $x = L$ boundary conditions are imposed for $R^-$.
Consequently, the information leaving the domain through the boundaries is the one that we are able to measure.
The boundary conditions \eqref{eq:physicalbcU} are transformed (after linearization of the control $U$) into Riemann coordinates in such a way that we consider the following boundary values:
\begin{equation}\label{eq:bc}
\begin{pmatrix}
R^+(t,0)\\R^-(t,L)
\end{pmatrix}=K
\begin{pmatrix}
R^+(t,L)\\R^-(t,0)
\end{pmatrix}.
\end{equation}
The explicit expression of $K$ depends on 
%the transformation matrix
$T$ as well as the linearized control $U$ and the physical boundary conditions $\A$ and $\B$.
Later, we will give an explicit expression for $K$.
For every initial condition, $R_0\in L^2((0,L);\R^n)$, the existence and uniqueness of an $L^2$-solution in the sense of \cite[Definition A.3]{bastin2016stability} to \eqref{eq:Riemann_sys}--\eqref{eq:bc} is guaranteed by \cite[Theorem A.4]{bastin2016stability}.

The stabilization of \eqref{eq:state_sys} around a zero steady state in most of the literature \cite{bastin2016stability,coron2008dissipative,banda2020numerical,bastin2008using,diagne2011lyapunov} is obtained by studying the system \eqref{eq:Riemann_sys} and the boundary conditions \eqref{eq:bc}.
This results in imposing sufficient conditions on the matrix $K$.
Stability is understood in the following sense:
\begin{definition}\label{def:expstable}
The linear system \eqref{eq:Riemann_sys}--\eqref{eq:bc} is called
\begin{enumerate}
\setlength{\itemsep}{0em}
    \item[(i)] exponentially stable in the $L^2$ norm if there exist $\nu>0$ and $C>0$ such that for every initial condition $R_0\in L^2((0,L);\R^n)$, the $L^2$-solution to the Cauchy problem \eqref{eq:Riemann_sys}--\eqref{eq:bc} satisfies
     \[ \Vert R(\cdot,t)\Vert_{L^2((0,L);\R^n)}\leq C\exp{(-\nu t)}\Vert R_0\Vert_{L^2((0,L);\R^n)},\quad \forall\ t\in[0,+\infty).\]
\item[(ii)] robustly exponentially stable if there exists $\varepsilon>0$ such that the perturbed system
    \[\partial_t R(t,x)+ \tilde \Lambda \partial_x R(t,x)=0,\]
    is exponentially stable for every diagonal matrix $\tilde \Lambda$ with $|\tilde \lambda_i-\lambda_i|\leq \varepsilon\ \forall\ i\in\{1,\dots,n\}$  where $\lambda_i$ are diagonal components of matrix $\Lambda$ in \eqref{eq:Riemann_sys}.
\end{enumerate}
\end{definition}
% \begin{definition}\label{def:expstable}
%     The linear system \eqref{eq:Riemann_sys} with boundary conditions \eqref{eq:bc} is called exponentially stable in the $L^2$ norm if there exist $\nu>0$ and $C>0$ such that for every initial condition $R_0\in L^2((0,L);\R^n)$, the $L^2$-solution to the Cauchy problem \eqref{eq:Riemann_sys}--\eqref{eq:bc} satisfies
%     \[ \Vert R(\cdot,t)\Vert_{L^2((0,L);\R^n)}\leq C\exp{(-\nu t)}\Vert R_0\Vert_{L^2((0,L);\R^n)},\quad \forall\ t\in[0,+\infty).\]
% \end{definition}
\noindent Thus, exponential stability means that the solution $R(t,x)$ converges to zero exponentially in time.
Furthermore, if a system is robustly exponentially stable, it is exponentially stable.\\
To guarantee exponential stability the following function $\rho_1(K)$ has been introduced in \cite{coron2008dissipative}
\begin{align*}
\rho_1(K):=\inf \{ \Vert \Delta K \Delta^{-1}\Vert ; \Delta \in \mathcal{D}_{n,+} \},
\end{align*}
where $\Vert K\Vert$ is the matrix norm induced by the Euclidean norm and $\mathcal{D}_{n,+}$ denotes the set of diagonal $n\times n$ real matrices with strictly positive diagonal entries.
This results in a sufficient condition for exponential stability:
\begin{theorem}(The reader is referred to \cite[Theorem 3.2]{bastin2016stability})\label{thm:exp}
    The system \eqref{eq:Riemann_sys}--\eqref{eq:bc} is exponentially stable in the $L^2$ norm if $\rho_1(K)<1$.
\end{theorem}
\noindent For robust exponential stability, Silkowski \cite{silkowski1978star} investigated
\[ \rho_0(K):=\max \{\rho(\text{diag}(e^{i\theta_1},\dots,e^{i\theta_n})K); (\theta_1,\dots,\theta_n)^T\in\mathbb{R}^n\}.
\]
where $\rho$ is the spectral radius. 
Note that for $K\in \R^{n\times n}$, the inequality
        \[\rho_0(K)\leq \rho_1(K)\]
holds, see \cite{coron2008dissipative}.
Further, by \cite[Theorem 3.12]{bastin2016stability} $\rho_0(K)=\rho_1(K)$ holds for $K\in\R^{n\times n}$ and $n\leq 5$. 
In particular, the following corollary gives a necessary and sufficient condition for robust exponential stabilization:
\begin{corollary}\label{cor:rho0}(the reader is referred to \cite[Corollarly 3.10]{bastin2016stability}) The system \eqref{eq:Riemann_sys}--\eqref{eq:bc} is robustly exponentially stable with respect to the characteristic velocities if and only if $\rho_0(K)<1$.
\end{corollary}

\section{Restricted Observability}
Our goal is to consider the case of limited boundary information about the system, i.e., we are not able to observe the boundary values of all physical variables.
This situation occurs in several applications, because not all variables are or can be measured.
Hence, these variables cannot be used in the control, $U$.
It is important to mention that in applications the variables observed (or not) are the physical variables, $Y$, and not the Riemann invariants, $R$.
Therefore, we have to start from the physical boundary conditions and transform them into a similar setting as in \eqref{eq:bc}.
This gives sufficient conditions for the control in the physical variables.

\subsection{Boundary transformation}
We start by transforming the boundary conditions \eqref{eq:physicalbcU} into Riemann coordinates. 
We assume that after linearization these boundary conditions are well-defined and can be expressed as follows:
\begin{definition}(Boundary conditions) The boundary conditions are given by
\begin{equation}\label{eq:physicalbc}
\A Y(t,0) + \B Y(t,L)= \Ko Y(t,0) +\KL Y(t,L).
\end{equation}
with  $\A$, $\B$, $\Ko$, $\KL\in \R^{n\times n}$. The matrices $\A$ and $\B$ represent the physical boundary constraints. The matrices $\Ko$ and $\KL$ are control matrices which contain zero columns for the variables that are not observable.
\end{definition}
The tuning parameters in $\Ko$ and $\KL$ are used to drive the system to its equilibrium.
Here, we make the reasonable assumption that, if the value of a variable cannot be observed, it cannot be used as a control input, either.
On the contrary, the unobserved variables may still be present on the left side of \eqref{eq:physicalbc}, since the physical constraints at the boundaries can be satisfied without observing the variables.
For example, open water channels can be controlled by underflow or overflow gates.
Raising or lowering these gates has a direct effect on the velocity and water depth due to physical laws.
The values do not need to be known.

To use the established stability results, we transform the boundary conditions \eqref{eq:physicalbc} using $Y(t,x)=TR(t,x)$ using 
\begin{equation}\label{eq:bcRiemanngeneral}
 (\A-\Ko) T R(t,0) =(\KL-\B) T R(t,L).
\end{equation}
Since we assume that the eigenvalues are ordered as $\lambda_1>\dots>\lambda_{m}>0>\lambda_{m+1}>\dots>\lambda_{n}$, we can divide the matrix $T$ into the following block matrices
\[T=\begin{pmatrix}
\Too & \Tot\\ \Tto & \Ttt
\end{pmatrix}\]
with $\Too\in\R^{m\times m},\ \Tot\in\R^{m\times (n-m)},\ \Tto\in \R^{(n-m)\times m}$ and $\Ttt\in \R^{(n-m)\times (n-m)}$.
We divide the matrices $\A,\ \B, \Ko, \KL$ in a similar fashion.\\
Note that the boundary conditions \eqref{eq:physicalbc} in physical variables are well-defined if and only if the boundary conditions \eqref{eq:bcRiemanngeneral} can be solved with respect to $R^+(t,0)$ and $R^-(t,L)$, i.e. 
\begin{equation}\label{eq:bc:transformed}
\begin{aligned}&
\begin{pmatrix}
R^+(t,0) \\
R^-(t,L)
\end{pmatrix}
= \mathcal{C}^{-1} \mathcal{D}%\begin{pmatrix}
\begin{pmatrix}
R^+(t,L) \\
R^-(t,0)
\end{pmatrix}
\end{aligned}
\end{equation}
with the entries of $\mathcal{C}$ and $\mathcal{D}$ given by
\begin{align*}
 &   \mathcal{C}_{1,1}=(\Aoo-\Koo)\Too+(\Aot-\Kot)\Tto & \mathcal{C}_{1,2}=(\Boo-\KLoo)\Tot+(\KLot-\Bot)\Ttt\\
&\mathcal{C}_{2,1}=(\Ato-\Kto)\Too+(\Att-\Ktt)\Tto & \mathcal{C}_{2,2}=(\Bto-\KLto)\Tot+(\KLtt-\Btt)\Ttt\\
&\mathcal{D}_{1,1}=(\KLoo-\Boo)\Too+(\Bot-\KLot)\Tto & \mathcal{D}_{1,2}=(\Koo-\Aoo)\Tot+(\Kot-\Aot)\Ttt\\
&\mathcal{D}_{2,1}=(\KLto-\Bto)\Too+(\Btt-\KLtt)\Tto & \mathcal{D}_{2,2}=(\Kto-\Ato)\Tot+(\Ktt-\Att)\Ttt
\end{align*}
For stability, we need to check that $K=\mathcal{C}^{-1} \mathcal{D}$ fulfills the aforementioned stability conditions.
In the following, we will study this for specific structures of the boundary conditions.

\subsection{One velocity direction and $K$ of rank one}
We will illustrate this mechanism for a system with only positive eigenvalues.
Negative eigenvalues can be treated in a similar manner.
Since the velocities are all positive, it is reasonable to assume that the physical boundary conditions involve constraints only at the left boundary and that the control depends on the flow of the system at $L$ and the state at $0$.
Therefore, $\B$ is a zero matrix, and the boundary conditions are
$$ (\A-\Ko) Y(t,0)=\KL Y(t,L).$$ 
Recall that the weight matrices $\KL$ and $\Ko$ contain zero columns for the unobserved variables.
Let us further assume that the boundary conditions above are well-posed and hence, $\A-\Ko$ is invertible.
Then we have the following constraints in Riemann coordinates:
\begin{equation}\label{eq:bc_res_pos}
R^+(t,0)=((\A-\Ko) T)^{-1}\KL TR^+(t,L).
\end{equation}
and robust exponential stability is given if and only if
\begin{equation}\label{eq:cond_pos}
\rho_0(((\A-\Ko)T)^{-1}\KL T)<1.
\end{equation}
Obviously, if $\KL$ is a zero matrix, exponential stability is guaranteed.\\
In the following, we consider the special case in which the control matrix $\KL$ is of rank one. 
In such a case we can write $\KL=uv^T$ for some vectors $u,v\in \R^n$.
Note that in general, a rank one control might be a questionable choice for a control matrix, since it is a rather simple control.
However, the specific case of observing only one variable at $x=L$ is expressed by a rank one matrix $\KL$, which makes such matrices of particular interest in the restricted observability setting.
In particular, for a $2\times 2$-system a control matrix of rank one covers the possible cases.
\begin{example}
    Two examples of control matrices of rank one are the following:
    \begin{itemize}
        \item We observe a certain set of variables, but use a uniform weight across all observed variables as the inflow for each physical variable, but for the physical variables these weights may differ, i.e.
        \[u=\begin{pmatrix}
            k_1\\ \vdots \\k_n
        \end{pmatrix},\quad v=(v_i)_{i=1,\dots,n}\text{ with }v_i=\begin{cases}
            1,\quad &\text{if variable $i$ is observed}\\
            0& \text{else}
        \end{cases}\]
        and $\KL=uv^T$. Hence, the entries in the row $i$ of $\KL$ are either $k_i$ or 0.\\
        In particular, the case that we only observe one variable is contained in this setting.
        \item  On the contrary, we can observe a given set of variables and use different weights for each observed variable that are uniform across the physical variables, i.e.
                \[u=\begin{pmatrix}
            1\\ \vdots \\1
        \end{pmatrix},\quad v=(v_i)_{i=1,\dots,n}\text{ with }v_i=\begin{cases}
            k_i,\quad &\text{if variable $i$ is observed}\\
            0& \text{else}
        \end{cases}\]
        and $\KL=uv^T$. Hence, the column $i$ of $\KL$ has either only the same entries $k_i$ or is a zero column.
    \end{itemize}
\end{example}
Another advantage of a rank one matrix is that the necessary and sufficient condition \eqref{eq:cond_pos} can be computed explicitly:
we get the following proposition:
\begin{lemma}\label{lem:1}
    Let $K\in\R^{n\times n}$ with $K=uv^T$ be a matrix of rank one with $u,v$ nonzero column vectors. Then we have
    \begin{equation}
        \rho_0(K)=\rho_1(K)=|v|^T|u|.
    \end{equation}
\end{lemma}
\begin{proof}
    We start by proving the equality of $\rho_0(K)=|v|^T|u|$. We obtain
    \begin{align*}
        \rho_0(K)&=\max\{\rho( (\text{diag}(e^{i\theta_1},\dots, e^{i\theta_n})u v^T)); (\theta_1,\dots,\theta_n)^T\in\R^n\}\\
        &=\max\{| v^T\text{diag}(e^{i\theta_1},\dots, e^{i\theta_n})u|; (\theta_1,\dots,\theta_n)^T\in\R^n\}\\
        &=\max\left\{  \left(\left(\sum_{j=1}^n v_j u_j\cos(\theta_j)\right)^2+\left(\sum_{j=1}^n v_j u_j\sin(\theta_j)\right)^2\right)^{\frac12}; (\theta_1,\dots,\theta_n)^T\in\R^n\right\}\\
        &= \left(\left(\sum_{j=1}^n |v_j u_j|\frac{1}{\sqrt{2}}\right)^2+\left(\sum_{j=1}^n |v_j u_j|\frac{1}{\sqrt{2}}\right)^2\right)^{\frac12}\\
        &=\sum_{j=1}^n |v_j u_j|=|v|^T|u|.
    \end{align*}
    Note that we have $|v|^T|u|=\rho(|K|)$ (since the only eigenvalues of $|K|$ are $0$ and $|v|^T|u|$).
    By \cite[Proposition 3.2]{bastin2016stability} we know $\rho_1(K)\leq \rho(|K|)$. Hence, with the inequality $\rho_0(K)\leq \rho_1(K)$ the statement follows.
\end{proof}
Due to this lemma and corollary \ref{cor:rho0}, we have proven the following theorem
\begin{theorem}\label{cor:pos}
    Let $\KL=uv^T$ be of rank one, $\A-\Ko$ is invertible and let the characteristic velocities be positive. The system \eqref{eq:Riemann_sys} with boundary conditions \eqref{eq:bc_res_pos} is robustly exponentially stable if and only if
    \begin{equation}
        |v^T T|\cdot |((\A-\Ko) T)^{-1} u|<1.
    \end{equation}
\end{theorem}
This theorem gives a necessary and sufficient condition for a system to be robustly exponentially stable. 
Note that exponential stability is still possible if a system is not robustly exponentially stable, but by Lemma \ref{lem:1}, we are not able to draw any further conclusions using Theorem \ref{thm:exp}.
\begin{remark}
    Similar results can be obtained for only negative velocities and for the case in which $M$ in \eqref{eq:state_sys} is already a diagonal matrix.
    For simplicity, let $M=\Lambda$, then we have $Y=R$ and we can replace $R(t,x)$, if necessary, by
    \[\begin{pmatrix}
        R_+(t,x)\\
        R_-(t,L-x)
    \end{pmatrix},\]
    so that the characteristic velocities are now all positive.
    The results derived above follow immediately with the boundary conditions \eqref{eq:bc_res_pos}.
\end{remark}
\begin{example} We consider the following system which is already in Riemann coordinates
$$ \partial_t Y(t,x)+\begin{pmatrix}
    \lambda_1 & 0 \\
    0 & -\lambda_2
\end{pmatrix}\partial_x Y(t,x)=0,$$
$\lambda_1,\lambda_2>0$ and the boundary conditions given by \eqref{eq:physicalbc} with 
$$\A=\begin{pmatrix}
    1&0\\0&0
\end{pmatrix}\quad\text{and}\quad \B=\begin{pmatrix}
    0&0\\0&1
\end{pmatrix}.$$
With $\tilde Y(t,x):=(Y_1(t,x),Y_2(t,L-x))^T$ and $\tilde \A:=Id_2$ we can transform the system to a problem with only positive eigenvalues $\lambda_1$ and $\lambda_2$. 
Let only the first variable be observable. Hence, the control laws are
$$\Ko=\begin{pmatrix}
    \Koo&0\\ \Kto&0
\end{pmatrix}\quad\text{and}\quad \KL=\begin{pmatrix}
    \KLoo &0\\ \KLto&0
\end{pmatrix}.$$
By theorem \ref{cor:pos} the system is robustly exponentially stable if and only if
\[ \left| \frac{\KLoo}{1-\Koo}\right|<1.\]
Interestingly, there is no condition on the parameters $\KLto$ and $\Kto$ which influence the non-observable variable.
\end{example}

\subsection{Local controls}
Now let us return to the general case.
We assume that we only have local physical constraints and controls.
For example, at $x=0$ the physical constraints are only influenced by the state variables at $x=0$ and also only controlled by the variables at $x=0$.
This situation occurs in many applications where only local information is available.
One possible choice to model such local control is that $\Ato$, $\Att$, $\Boo$, $\Bot$, $\Kto,\ \Ktt,\ \KLoo$ and $\KLot$ are all zero matrices.
This results in a simpler structure of the matrix we need to check for stability, i.e.,
this results in a block anti-diagonal matrix for which we need to check for stability, i.e.,
\begin{align}\label{eq:bc_approach2}
&\begin{pmatrix}
R^+(t,0) \\
R^-(t,L)
\end{pmatrix}
=\begin{pmatrix}
0&  \mathcal{C}_{1,1}^{-1}\mathcal{D}_{1,2}\\
\mathcal{C}_{2,2}^{-1}\mathcal{D}_{2,1}& 0
\end{pmatrix}
\begin{pmatrix}
R^+(t,L) \\
R^-(t,0)
\end{pmatrix}.
\end{align}
\begin{remark}
 Note that even if $n$ is even and $m=n/2$, i.e., we have the same number of positive and negative eigenvalues, no identity matrices result from the multiplication of the block matrices, since $\Too\neq\Tot$ and $\Tto\neq\Ttt$ hold (since the inverse of $T$ exists).
\end{remark}
Especially for the case $n=2$, the above representation is very helpful, because the inverses are given by scalars and the stability conditions are easy to check.
In fact, all stability conditions introduced so far lead to the same result:
\begin{proposition}\label{prop:n2}
    Let $a,b\in\R$ and $$K=\begin{pmatrix}
0 & a\\
b & 0
\end{pmatrix}.$$
Then, we have
$$ \rho_0(K)=\rho_1(K)=\sqrt{|ab|}.$$
\end{proposition}
\begin{proof}
    The equality $\rho_1(K)=\rho_0(K)$ is given in \cite[Theorem 3.12]{bastin2016stability}. Further, straight-forward computations show that $\rho_0(K)=\sqrt{|ab|}$.
\end{proof}
In the following, boundary stabilization will be applied to an example with $n = 2$.

\subsubsection{Density-Flow Systems}
One prominent example for a linear system with $n=2$ is the so called density-flow system, i.e., 
\begin{equation}\label{eq:densityflow}
   \begin{aligned}
\partial_t H+\partial_x Q&=0\\
\partial_t Q+\lambda_1\lambda_2\partial_x H+(\lambda_1-\lambda_2)\partial_x Q&=0
\end{aligned} 
\end{equation}
with $\lambda_1,\ \lambda_2>0$ and boundary conditions on the flux $Q$, i.e.,
\[ Q(t,0)=U_0(t)\qquad Q(t,L)=U_L(t),\]
for some suitable control laws $U_0$ and $U_L$.
The eigenvalues of the system are given by $\lambda_1$ and $-\lambda_2$ and the matrix $T$ by
$$T=\frac{1}{\lambda_1+\lambda_2}\begin{pmatrix}
    1 & -1\\ \lambda_1&\lambda_2
\end{pmatrix}.$$
Such a system can be used to represent many physical systems, for example, it is a valid approximate linearized model for the motion of liquid fluids
in pipes.
For more details, refer to \cite[Section 2.2]{bastin2016stability} and \cite[Chapter 2]{nicolet2007hydroacoustic}.\\
The boundary conditions \eqref{eq:physicalbc} can be obtained after linearizing around a constant equilibrium $H^*$ and $Q^*$ and applying linear feedback law, such that we obtain
\begin{align*}
Q(t,0)-Q^*&=\Koo (H(t,0)-H^*)+\Kot (Q(t,0)-Q^*)\\
Q(t,L)-Q^*&=\KLto (H(t,L)-H^*)+\KLtt (Q(t,L)-Q^*).
\end{align*}
Note that the controls are only scalars.\\
If both variables at each boundary are observable, the system is robustly exponentially stable (see proposition \ref{prop:n2} and corollary \ref{cor:rho0}) if and only if
\begin{align*}
\rho_0(K)^2=\left| \frac{(-\Koo+(\Kot-1)\lambda_2)(\KLto+(\KLtt-1)\lambda_1)}{(-\Koo+(1-\Kot)\lambda_1)(\KLto+(1-\KLtt)\lambda_2)}\right|<1.
\end{align*}
In case, we only observe the flow $Q$ at the boundaries, the system is not robustly exponentially stable as
\begin{align*}
\rho_0(K)^2=\left| \frac{(\Kot-1)\lambda_2(\KLtt-1)\lambda_1}{(1-\Kot)\lambda_1(1-\KLtt)\lambda_2}\right|=1.
\end{align*}
Hence, a restricted observability can make a robust feedback control impossible.
%it impossible to stabilize a system towards its equilibrium.
In fact, it depends on observing the correct boundaries, as observing only $H$ gives us
\begin{align*}
\left| \frac{(-\Koo-\lambda_2)(\KLto-\lambda_1)}{(-\Koo+\lambda_1)(\KLto+\lambda_2)}\right|<1
\end{align*}
as a stability condition. Interestingly, it is sufficient to only observe one of the two boundaries, for example, for observing only the boundary at $x=0$ stability is given for
\begin{align*}
\left| \frac{(-\Koo-\lambda_2)\lambda_1)}{(-\Koo+\lambda_1)\lambda_2)}\right|<1.
\end{align*}
Let us consider a tuning parameter with $\Koo<0$ as an example. For $\Koo\in [-\lambda_2,0)$ the system is robustly exponentially stable. 
If the eigenvalues additionally satisfy $\lambda_1<\lambda_2$, we can choose any $\Koo<0$ and the system remains robustly exponentially stable.

\subsubsection{Numerical example}
To further illustrate the results for the density-flow system, we will consider a numerical example.
We choose the initial conditions
\[H_0(x)=2.5,\qquad Q_0(x)=4\sin(\pi x).\]
An equilibrium for the system \eqref{eq:densityflow} is given by an arbitrary constant state $H^*$ and $Q^*$. Here we choose 
\[H^*=2,\qquad Q^*=3.\]
We also set $\lambda_1=1<\lambda_2=2$.
We solve the problem on the interval $[0,1]$.
The solutions are approximated numerically by solving the system in Riemann coordinates with a simple upwind scheme.
%\jan{We can also use a Lax-Friedrichs-type scheme for the physical variables and a characteristic decomposition for the boundary values. I have the code for this, too.}
The spatial step size is set to $\Delta x=0.01$ and the temporal step size is set to $0.75\Delta t\leq \Delta x/\max\{\lambda_1,\lambda_2\}$ according to the CFL condition.
We compute the discrete $L^2-$norm of the difference between the physical variables and  their equilibrium for different parameter settings. In particular, we assume that the flux is unobservable and consider a control for the height only at $x=0$, i.e. $\Kot=\KLto=\KLtt=0$.
From the analysis in the previous subsection, we know that robust stabilization is guaranteed for any $\Koo<0$.
In addition, if neither the height nor the flux is observable, the system is not robustly exponentially stable since $\rho_0(K)=1$.
Therefore, we study $\Koo\in\{0.01,0,-0.01,-1,-100\}$.
The $L^2-$norm in figure \ref{fig:L2densityflow} shows that exponential stability can be obtained for the negative values of $\Koo$.
\begin{figure}
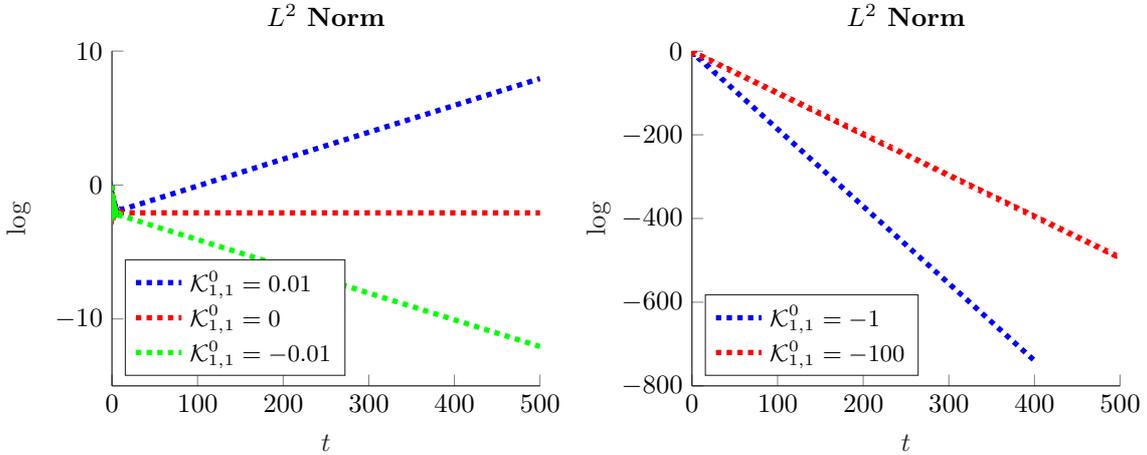

    \centering
    \setlength{\fwidth}{0.37\textwidth}
    \input{L2NormDF1}
    \input{L2NormDF2}
    \caption{Discrete $L^2-$norm on a logarithmic scale for $\Koo\in\{0.01,0,-0.01\}$ (left) and $\Koo\in\{-1,-100\}$ (right).}
    \label{fig:L2densityflow}
\end{figure}
For the first two cases, the system is not only robustly exponentially stable (as known from the analytical results), but also not exponentially stable at all.
Note that observing the flow always leads to the same boundary conditions in Riemann coordinates, since they depend only on $\lambda_1$ and $\lambda_2$. Thus the example above for $\Koo=0$ covers all possible cases. In particular, the system here stabilizes to a different steady state.

\subsection{Balance laws}
The previous approaches can be easily extended to linear balance laws, i.e.,
\begin{equation}\label{eq:balancelaw}
\partial_t Y(t,x)+ M\partial_x Y(t,x)+NY(t,x)=0,\quad t>0,\ x\in [0,L]
\end{equation}
with $N\in\R^{n\times n}$.
Here, we obtain in Riemann coordinates,
\begin{equation}\label{eq:balancelawRC}
\partial_t R(t,x)+\Lambda \partial_x R(t,x) +S R(t,x)=0,\quad t>0,\ x\in [0,L]
\end{equation}
with $S=T^{-1} N T$.
The exponential stability in the $L^2$ norm now involves additional conditions for the source term.
To the best of the authors' knowledge, only sufficient conditions for exponential stability are available.
Here, we recall the following result from \cite{bastin2016stability}:
\begin{theorem}\label{thm:bl}(\cite[Theorem 5.4]{bastin2016stability})
     If there exists $P=\diag(p_1,\ldots,p_n)$ with $p_i>0$ for $i=1,\ldots,n$, such that
\[S^T P +P S\text{ is positive semi-definite}\]
and 
\[\Vert \Delta K \Delta^{-1} \Vert <1 \]
with $\Delta :=\sqrt{P|\Lambda|}$, then the system \eqref{eq:balancelawRC} with boundary conditions \eqref{eq:bc} is exponentially stable in the sense of Definition \ref{def:expstable} (i).
\end{theorem}
Obviously, for our purposes, $K$ is given by a suitable transformation of the physical boundary conditions as in \eqref{eq:bc:transformed}.
We will present an example in the following.

\subsubsection{Shallow-Water equations}
Let us consider the Saint-Venant or Shallow-Water equations for an open channel, see e.g. \cite{bastin2008using} or \cite[Section 1.4]{bastin2016stability}.
In particular, we consider the water flow along a prismatic channel with a rectangular cross-section, a length of $L$ units and a constant bottom slope.
The system is given by
\begin{align*}
\partial_t H+\partial_x(HV)&=0\\
\partial_t V+\partial_x\left(\frac{V^2}{2}+gH\right)+\left(C\frac{V^2}{H}-gS_b\right)&=0,
\end{align*}
where the physical quantity $H$ is the water depth and $V$ is the horizontal water velocity.
The parameters for the source term are the constant bottom slope $S_b$, the constant gravity acceleration $g$ and a constant friction coefficient $C$.
We consider a channel on the interval $[0,L]$ and assume that we only use local controls at each gate.
For simplicity, we consider the case that the pool is closed and endowed with pumps at $x=0$ and $x=L$ to impose discharges.
This gives the boundary conditions
\begin{align}\label{eq:bc_SV}
H(t,0)V(t,0)=U_0(H(t,0),V(t,0));\quad H(t,L)V(t,L)=U_L(H(t,L),V(t,L))
\end{align}
with suitable control laws $U_0$ and $U_L$, see \cite[Section 1.4.1]{bastin2016stability}.
Note that we are choosing only local control laws.
We consider the flow in the subcritical regime, i.e.
\begin{align*}
\frac{V(t,x)}{\sqrt{g\, H(t,x)}}<1,
\end{align*}
such that the system is hyperbolic. Linearizing the equations around a suitable constant equilibrium $V^*$ and $H^*$, we obtain \eqref{eq:balancelaw} with
\begin{align*}
Y=\begin{pmatrix}
h\\
v
\end{pmatrix},\quad
M=\begin{pmatrix}
V^*&H^*\\g&V^*
\end{pmatrix},\quad\text{and}\quad
N=\begin{pmatrix}
0&0\\-\frac{C{V^*}^2}{{H^*}^2} & \frac{2CV^*}{H^*}
\end{pmatrix},
\end{align*}
where $h(t,x)=H(t,x)-H^*$ and $v(t,x)=V(t,x)-V^*$. Linearizing the boundary conditions \eqref{eq:bc_SV} and using the linear control given in \eqref{eq:physicalbc}, results in
\begin{align*}
\A =\begin{pmatrix}
V^*&H^*\\ 0&0
\end{pmatrix},\quad
\B=\begin{pmatrix}
0&0\\V^*&H^*
\end{pmatrix}, \quad\Ko=\begin{pmatrix}
\Koo&\Kot\\
0&0
\end{pmatrix}\quad\text{and}\quad \KL\begin{pmatrix}
0&0\\
\KLto&\KLtt
\end{pmatrix}.
\end{align*}
The corresponding matrix $T$ is given by
\[T=\frac12\begin{pmatrix}
\sqrt{\frac{H^*}{g}}&-\sqrt{\frac{H^*}{g}}\\
1&1
\end{pmatrix}\]
with the eigenvalues $\lambda_1=V^*+\sqrt{gH^*}>0$ and $\lambda_2=V^*-\sqrt{gH^*}<0$.
We restrict ourselves to equilibria which fulfill the following additional condition:
\[gS_b{H^*}=C{V^*}^2.\]
In such a way the linearized source term in \eqref{eq:balancelawRC} simplifies to
\begin{align*}
S=\begin{pmatrix}
\gamma &\delta\\ \gamma & \delta
\end{pmatrix}\quad \text{with}\quad \gamma=\frac{C{V^*}^2}{H^*}\left(\frac{1}{V^*}-\frac{1}{2\sqrt{gH^*}}\right)\quad \text{and}\quad \delta=\frac{C{V^*}^2}{H^*}\left(\frac{1}{V^*}+\frac{1}{2\sqrt{gH^*}}\right).
\end{align*}
To fulfill theorem \ref{thm:bl}, we can choose $P=\text{diag}(\delta,\gamma)$ such that the first sufficient condition for stabilization in theorem \ref{thm:bl} is guaranteed.
This enforces $\Delta=\text{diag}(d_1,d_2)$ with $d_1=\sqrt{\delta \lambda_1}$ and $d_2=\sqrt{\gamma |\lambda_2|}$.
We obtain $\Vert \Delta \tilde K \Delta^{-1}\Vert<1$ with $\tilde K$ given by \eqref{eq:bc_approach2} is equivalent to
\begin{align*}
\max\left\{
{\frac{d_1}{d_2}}\left|
\frac{(V^*-\Koo)\sqrt{\frac{H^*}{g}}-(H^*-\Kot)}{(V^*-\Koo)\sqrt{\frac{H^*}{g}}+(H^*-\Kot)}\right|,
{\frac{d_2}{d_1}}\left|\frac{(V^*-\KLto)\sqrt{\frac{H^*}{g}}+(H^*-\KLtt)}{(V^*-\KLto)\sqrt{\frac{H^*}{g}}-(H^*-\KLtt)}\right|
\right\}<1.
\end{align*}
Note that $d_1>d_2$ holds.
For simplicity, let us choose the following ranges for possible controls: $\Koo,\KLto\in(0,V^*)$ and $\Kot,\KLtt\in(0,H^*)$. Both controls need to be chosen such that the condition above is fulfilled. We concentrate on the right boundary first.\\
If we observe both physical variables, we can choose $\KLtt$ or $\KLto$ first, which gives us then a suitable choice for the remaining parameter.
For $\KLto\in(0,V^*)$, we are left with choosing either
\[\KLtt\in\left(H^*-\frac{d_1-d_2}{d_1+d_2}(V^*-\KLto)\sqrt{\frac{H^*}{g}},H^*\right)\]
or, if the set is not empty, e.g. for $\KLto>V^*-\frac{d_1-d_2}{d_1+d_2}\sqrt{gH^*}$,
\[\KLtt\in\left(0,H^*-\frac{d_1+d_2}{d_1-d_2}(V^*-\KLto)\sqrt{\frac{H^*}{g}}\right).\]
On the contrary, by choosing $\KLtt\in(0,H^*)$ first, we can choose either
\[\KLto\in\left(V^*-\frac{d_1-d_2}{d_1+d_2}(H^*-\KLtt)\sqrt{\frac{g}{H^*}},V^*\right)\]
or, again if the set is not empty, e.g. for $\KLtt>H^*-\frac{d_1-d_2}{d_1+d_2}V^*\sqrt{\frac{H^*}{g}}$,
\[\KLto\in\left(0,V^*-\frac{d_1+d_2}{d_1-d_2}(H^*-\KLtt)\sqrt{\frac{g}{H^*}}\right).\]
It can already be noted that in case one control is small the other one needs to be chosen large.
In particular, if we only observe one of the physical quantities, it enforces the following:
\begin{align*}
&\KLto=0\rightarrow\KLtt\in\left(H^*-\frac{d_1-d_2}{d_1+d_2}V^*\sqrt{\frac{H^*}{g}},H^*\right)\\
&\KLtt=0\rightarrow\KLto\in\left(V^*-\frac{d_1-d_2}{d_1+d_2}H^*\sqrt{\frac{g}{H^*}},V^*\right).
\end{align*}
In both cases, the control laws can be chosen such that the exponential stability of the systems can be guaranteed (given a suitable choice for the left boundary values at $x=0$), eventhough the choice is limited.
If $\KLto=\KLtt=0$, the norm is greater than one and stabilization cannot be guaranteed.\\
The left boundary behaves differently. Observing both boundaries, stabilization is guaranteed either by choosing
\[\Koo\in(0,V^*)\text{ and }\Kot\in\left(\max\left\{H^*-\frac{d_1+d_2}{d_1-d_2}\sqrt{\frac{H^*}{g}}(V^*-\Koo),0\right\},H^*-\frac{d_1-d_2}{d_1+d_2}\sqrt{\frac{H^*}{g}}(V^*-\Koo)\right)\]
or
\[\Kot\in(0,H^*)\text{ and }\Koo\in\left(\max\left\{V^*-\frac{d_1+d_2}{d_1-d_2}\sqrt{\frac{g}{H^*}}(H^*-\Kot),0\right\},V^*-\frac{d_1-d_2}{d_1+d_2}\sqrt{\frac{g}{H^*}}(H^*-\Kot)\right).\]
For this boundary, we also obtain the sufficient condition for stabilization, if $\Koo=0$ or $\Kot=0$. In both cases, the other tuning parameter needs to be chosen according to above. In the case $\Koo=\Kot=0$ the sufficient condition for stability is always met.\\
Hence, for this kind of boundary control, i.e., using pumps, it is sufficient to observe only the velocity or the depth of the water at $x=L$ to guarantee exponential stabilization towards the equilibrium.

\subsubsection{Numerical example}
We consider the Shallow-Water equations with the same parameters as used in \cite[Section 4.1]{banda2020numerical}:
The initial conditions are
\[H_0(x)=2.5,\qquad V_0(x)=4\sin(\pi x),\]
with the steady state
\[H^*=2,\qquad V^*=3.\]
We only consider a control at the boundary $x=L$, hence $\Koo=\Kot=0$.
We study different choices for the remaining parameters; see Table \ref{tab:para}.
\begin{table}[ht!]
    \centering
    \begin{tabular}{c|cccccc}
    Parameter-set& 1&2&3&4&5&6\\
    \hline
       $\KLto$  &   0&2.5&2.5&0&0&1\\
        $\KLtt$ &  1.75&0&0.5&0&1&0
    \end{tabular}
    \caption{Choices of boundary conditions. The sets 1-3 guarantee exponential stability.}
    \label{tab:para}
\end{table}
We solve the problem on the interval $[0,1]$.
The spatial step size is set to $\Delta x=0.01$ and the temporal step size is set to $0.75\Delta t\leq \Delta x/\max\{\lambda_1,\lambda_2\}$ according to the CFL condition.
As before, we approximate the solutions (in Riemann coordinates) by an upwind scheme and consider the discrete $L^2-$norm of the difference of the physical variables from their equilibrium.
The results for the different parameter sets can be seen in figure \ref{fig:Lyapunov}.
\begin{figure}[ht!]
    \centering
    \setlength{\fwidth}{0.75\textwidth}
    \input{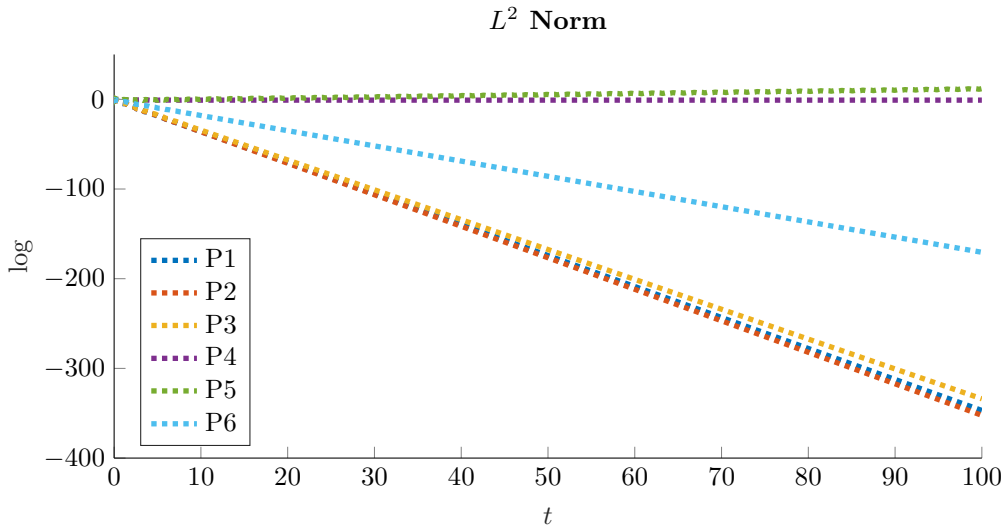}
    \caption{Discrete $L^2-$norm on a logarithmic scale for the different parameter sets of Table \ref{tab:para}}
    \label{fig:Lyapunov}
\end{figure}
It can be seen that the parameter choices one to three give exponential convergence and are in agreement with the theoretical results. 
The rate is very similar, even though the control of the boundary is different (once in velocity, once in depth, once in both).
Interestingly, controlling only one physical variable gives a slightly better rate than controlling two. 
For the parameters of set six, we obtain exponential convergence, too.
In contrast, it is obvious that the parameter sets four and five do not converge to the (correct) equilibrium.

\section{Conclusion}
In this work, we have shown how systems can be stabilized even with limited information. 
This case is of great importance for practitioners, since in many applications only limited information about the boundaries is available.
Furthermore, this information is given in the physical quantities.
Therefore, we decided to start from a boundary control in physical variables and to study the influence of limited information.
We studied linear systems of conservation and balance laws, which can be used to model many physically relevant problems.
In particular, for the case of positive characteristic velocities and the observation of only one physical quantity, we were able to establish a direct formula to compute the restrictions on the tuning parameters.

Future work may involve extending the results obtained here to quasilinear conservation laws, which would include nonconstant steady states.

\section*{Acknowledgments}
M.~H. received funding from the European Union's Horizon Europe research and innovation programme under the Marie Sklodowska-Curie Doctoral Network Datahyking (Grant No. 101072546).
J.~F. is supported by the German Research Foundation (DFG) through SPP 2410 'Hyperbolic Balance Laws in Fluid Mechanics: Complexity, Scales, Randomness' under grant FR 4850/1-1.
M.~K.~B. discloses receipt of the following financial support for the research authorship and publication of this article: this work was supported by the National Research Foundation of South Africa [grant no. CPRR23042296055].
%%%%%%%%%%%%%%%%%%%%%%%%%%%%%%%%%%%%%%

%\begin{thebibliography}{99}
%\bibliographystyle{ws-procs9x6}
\bibliographystyle{siam}
\bibliography{Bib_IncompleteBC}

\end{document}